\documentclass[12pt, a4paper]{amsart}
\usepackage{amsfonts}
\usepackage{amscd,amsmath}
\usepackage{tikz-cd}
\usepackage{amssymb}
\usepackage{setspace}
\newtheorem{theorem}{Theorem}[section]

\newtheorem{corollary}[theorem]{Corollary}
\newtheorem{proposition}[theorem]{Proposition}
\newtheorem{definition}[theorem]{Definition}
\newtheorem{example}[theorem]{Example}

\theoremstyle{remark}
\newtheorem{remark}{Remark}[section]
\newtheorem*{rem*}{Remark}

\newcommand{\comment}[1]{}
\numberwithin{equation}{section}
\begin{document}

\title{Multiplicative Lie Algebra Structures on a Group}
\author{Mani Shankar Pandey$^1$ and Sumit Kumar Upadhyay$^2$\vspace{.4cm}
}
\address{$^{1,2}$Department of Applied Sciences,\\ Indian Institute of Information Technology, Allahabad}
\email{\tiny{$^1$manishankarpandey4@gmail.com, $^2$upadhyaysumit365@gmail.com}}.

\begin{abstract} 
The main aim of this paper is to classify the distinct multiplicative Lie algebra structures (up to isomorphism) on a given group. We also see that for  a given group $G$, every homomorphism from the non abelian exterior square $G \wedge G$ to $G$ gives a multiplicative Lie algebra structure on $G$ under certain conditions. 
\end{abstract}
\maketitle

\section{Introduction}
In a group $(G,\cdot)$, the following commutator identities holds universally: 
\begin{align*}
\\&[a, a] = 1,\\&
[a,bc] = [a,b]^b[a,c],\\&
 [ab,c] = ^a[b,c][a,c],\\&
 [[a,b],^b c] [[b,c],^c a][c,a],^ab] = 1,\\&
 ^c[a,b]= [^ca,^cb],
\end{align*}
for every $a,b,c\in G$. In 1993, G. J. Ellis \cite{GJ} conjectured that these five universal identities generate all universal identities between commutators of weight $n$. For  $n=2$ and $3$, he gave an affirmative answer of this conjecture using homological techniques and introduced a new algebraic concept ``multiplicative Lie algebra". More precisely a multiplicative Lie algebra is a group $(G,\cdot)$ together with an extra binary operation $*$, termed as multiplicative Lie algebra structure satisfies identities similar to all five universal commutator identities.

For a given group $(G,\cdot)$, there is always a trivial multiplicative Lie algebra structure $``*"$ given as $a*b=1$, for all $a,b \in G$. In this case, $G$ is known as the trivial multiplicative Lie algebra. On the other hand if we consider a non-abelian group $G$, then there is atleast two distinct multiplicative Lie algebra structures on $G$, one is trivial and other one is given by the commutator, that is, $a*b=[a,b]$, for all $a,b \in G$. In this manner, it is interesting to know that how many distinct multiplicative Lie algebra structure (up to isomorphism) can exist on a group.

Let $G={\mathbb{Z}_p\times \mathbb{Z}_p \times \cdots \times \mathbb{Z}_p} ~(n-\text{copies})$ be an abelian group of order $p^n$, where $p$ is a prime. Since $G$ is a vector space over $\mathbb{Z}_p$, the multiplicative Lie algebra structure on $G$ is same as the Lie algebra structure on $G$. Thus for such type of abelian groups, the number of distinct multiplicative Lie algebra structures is same as number of distinct Lie algebra structures on $G$. In particular for $n=2$,  $G=\mathbb{Z}_p\times \mathbb{Z}_p=\langle a,b| a^p=b^p=1,ab=ba \rangle$ has two distinct multiplicative Lie algebra structures given by $a*b=1$ and $a*b=a$.   

We call a multiplicative Lie algebra $(G,\cdot, *)$ is Lie simple if $*$ is not given by $a*b=1$ or $a*b=[a,b]$. In this direction, we have following results from \cite{RLS}:
 \begin{proposition} \label{Lie simple 1} 
Every simple group is Lie simple.
\end{proposition}
\begin{proposition} \label{Lie simple 2}
Every cyclic group is Lie simple. Evidently, an abelian group is Lie simple if and only if it is cyclic.
\end{proposition}
\begin{proposition} \label{Lie simple 3}
Every free group is Lie simple.
\end{proposition}
\begin{proposition} \label{Equivariant homomorphism}
Let $G$ be a group with trivial Schur multiplier. Then any multiplicative Lie algebra structure on $G$ is uniquely determined by a $G-$equivariant homomorphism from $[G,G]$ to $G$.
\end{proposition}
In this paper, we prove that every multiplicative Lie algebra structure $(*)$ on a group $G$ can be uniquely determined by a quotient group of the exterior square $G\wedge G$ of $G$ under certain condition. We also determine the number of distinct (upto isomorphism) multiplicative Lie algebra structures on some finite groups like $\text{Sym}(n)$, $A_n$, $D_n$ and $Q_{n}$ etc.

Throughout the paper, we say that two multiplicative Lie algebra homomorphism $\phi_1,\phi_2:G_1\to G_2$ are distinct if ideals $\phi_1(G_1)$ and $\phi_2(G_1)$ are not isomorphic and also we say that two multiplicative Lie algebra structures $(*_1)$ and $(*_2)$ are distinct on a group $G$ if $G *_1 G$ and $G *_2 G$ are not isomorphic.
\section{Preliminary results}
In this section, we recall some basic definitions and results related to Schur multiplier which is essential for this paper. For $x, y \in G$, we will denote the commutator $xyx^{-1}y^{-1}$ by $[x, y]$ and the conjugate $yxy^{-1}$ by $^y x$.
\begin{definition} \cite{GJ}\label{MLA}
A multiplicative Lie algebra is a triple $(G, \cdot, *)$, where $(G, \cdot)$ is a group and $*$ is a binary operation on $G$ satisfying 
\begin{enumerate}
\item $a* a = 1$
\item $a*(b\cdot c) = (a*b)\cdot ^b(a*c)$
\item $(a\cdot b)*c = ^a(b*c)\cdot (a*c)$
\item $((a*b)*^b c)\cdot ((b*c)*^c a)\cdot ((c*a)*^ab) = 1$ (Jacobi identity)
\item $^c(a*b)= ^ca*^cb$
\end{enumerate}
for all $a,b,c \in G$.

The ideal generated by $\{a* b ~\mid ~ a, b \in G\}$ is denoted by $G* G$.
\end{definition}
\begin{definition}
Let $(G,\cdot,*)$ be a multiplicative Lie algebra.
\begin{enumerate}
\item A subgroup $H$ of $G$ is said to be subalgebra of G if $x*y \in H$ for all $x,y \in H$.
\item A subalgebra $H$ of $G$ is said to be an ideal of $G$ if it is a normal subgroup of $G$ and $x*y \in H$ for all $x \in G$ and $y \in H$. The ideal generated by $\{a* b ~\mid ~ a, b \in G\}$ is denoted by $G* G$.
\item Let $(G',\circ , *')$ be another multiplicative Lie algebra. A group homomorphism $\chi : G \longrightarrow G'$ is called a multiplicative Lie algebra homomorphism if $\chi(x*y) = \chi(x) *' \chi(y)$ for all $x, y \in G$.

\end{enumerate}
\end{definition}
\begin{definition}
Let $(G,\cdot)$ be a group. Then the non-abelian exterior square $G\wedge G$ of $G$ is a group generated by the elements of the set $\lbrace a\wedge b:~a,b \in G\rbrace$ satisfying the following conditions:
\begin{enumerate}
\item $a\wedge a=1$
\item $(a\wedge b)(b\wedge a)=1$
\item $ab\wedge c=(^ab\wedge ^ac )(a\wedge c)$
\item $a\wedge bc=(a \wedge b)(^ba \wedge ^bc)$
\end{enumerate}
for all $a,b,c \in G$. 
\end{definition}
\begin{theorem} \cite{RDL} \label{Tensor product of $D_n$}
Let  $D_n$ be the dihedral group of order $2n$. Then the non-abelian tensor product $D_n \bigotimes D_n$ of $D_n$ is given by
 \[   
D_n \bigotimes D_m = 
     \begin{cases}
       \mathbb{Z}_2 \times \mathbb{Z}_n  &\quad :n \ \text{is odd}\\
       \mathbb{Z}_2 \times \mathbb{Z}_n \times \mathbb{Z}_2 \times \mathbb{Z}_2 &\quad :n \ \text{is even} \\ 
     \end{cases}
     \]
where the factors are generated by $(x \bigotimes x)$ and $(x \bigotimes y)$ for odd $n$ and for even $n$ are generated by $(x \bigotimes x)$, $(x \bigotimes y)$, $(y \bigotimes y)$ and $(x\bigotimes y)(y \bigotimes x)$ respectively.     
\end{theorem}
\begin{theorem} \cite{KP} \label{tensor product}
The tensor product of $\mathbb{Z}_m$ and $\mathbb{Z}_n$ is isomorphic to $\mathbb{Z}_{gcd(m,n)}$.
\end{theorem}
\begin{theorem} \cite{SH}\label{Structure on 8p}
Let $G$ be a non abelian group of order $8p$, where $p$ is an odd prime. Then $G$ is isomorphic to exactly one group of the following types:
\begin{enumerate}
\item $D_4\times \mathbb{Z}_p$ \label{1}
\item $Q_2 \times \mathbb{Z}_p$ \label{2}
\item $D_{2p}\times \mathbb{Z}_2 $\label{3}
\item $Q_{p}\times \mathbb{Z}_2$\label{4}
\item $D_p\times \mathbb{Z}_4$\label{5}
\item $\langle a,b:a^8=b^p=1,~a^{-1}ba=b^{-1}\rangle$ \label{6}
\item $D_{4p}$\label{7}
\item $Q_{2p}$\label{8}
\item $\langle a,b,c:a^4=b^2=c^p=1,~b^{-1}ab=a^{-1},a^{-1}ca=c^{-1},bc=cb \rangle$ \label{9}
\item $\langle a,b:a^8=b^p=1,~a^{-1}ba=b^\alpha \rangle$, where $\alpha$ is a primitive root of $\alpha^4 \equiv 1~\text{mod}(p)$, $4$ divides $p-1$. \label{10}
\item $\langle a,b,c:a^4=b^2=c^p=1,~ab=ba,a^{-1}ca=c^\alpha, bc=cb \rangle$, where $\alpha$ is a primitive root of $\alpha^4 \equiv 1~\text{mod}(p)$, $4$ divides $p-1$.\label{11}
\item $\langle a,b:a^8=b^p=1,~a^{-1}ba=b^\alpha\rangle$, where $\alpha$ is a primitive root of $\alpha^8 \equiv 1~\text{mod}(p)$, $8$ divides $p-1$.\label{12}
\item $\mathbb{Z}_2 \times A_4$\label{13}
\item $\text{SL}(2,3)$\label{14}
\item $\text{Sym}(4)$\label{15}
\item $\langle a,b,c,d:a^4=b^2=c^2=d^p=1,~ab=ba, ac=ca, bc=cb,d^{-1}ad=b, d^{-1}bd=c,d^{-1}cd=ab \rangle $\label{16}
 \end{enumerate} 
\end{theorem}
\begin{theorem} \cite{KP} \label{Schur multiplier of p cyclic groups}
Let $G$ be a finite group. If the Sylow p-subgroups of $G$ are cyclic for all $p$ divides order of $G$, then the Schur multiplier $M(G)$ is trivial.
\end{theorem}
\begin{theorem} \cite{RL} \label{Schur multiplier 2}
Let $G$ and $H$ be two finite groups. Then the Schur multiplier
\begin{center}
 $M(G \times H)\cong M(G)\times M(H)\times (G^{ab} \bigotimes H^{ab})$,
 \end{center}
\end{theorem}
\begin{theorem} \cite{KP} \label{Schur multiplier of $D_n$}
The Schur multiplier of the dihedral group $D_n$ is given by
\[   
M(D_n) = 
     \begin{cases}
       1  &\quad :n \ \text{is odd}\\
       \mathbb{Z}_2  &\quad :n \ \text{is even} \\ 
     \end{cases}
\]
\end{theorem}
\begin{theorem} \cite{KP}
Let $\text{Sym}(n)$ be the symmetric group on $n$ symbols. Then the Schur multiplier:
\[   
M(\text{Sym}(n)) = 
     \begin{cases}
       1  &\quad :n\leq 3\\
       \mathbb{Z}_2  &\quad :n\geq 4 \\ 
     \end{cases}
\]
\end{theorem}
\begin{theorem} \cite{RL}\label{Schur multiplier of quaternion group}
The Schur multiplier of quaternion group $Q_{n}=\langle x,y:x^2=y^{n},xyx^{-1}=y^{-1}\rangle$ of order $4n$ is trivial.
\end{theorem}
\begin{theorem} \cite{SNAN} \label{Schur Multiplier of 8q}
Let $G$ be a non abelian group of order $8p$, where $p$ is an odd prime. Then the Schur multiplier
\[
  M(G) = \left\{
     \begin{array}{@{}l@{\thinspace}l}
       1  &: G~\text{is of type}~(2),(6),(8),(10),(12),(14),~\text{or}~ (16)\\
       \mathbb{Z}_2 &: G~\text{is of type}~(1),(4),(5),(7),(9),(11),(13),~\text{or}~(15) \\
        \mathbb{Z}_2\times \mathbb{Z}_2 \times \mathbb{Z}_2  &: G~ \text{is of type}~(3) \\

     \end{array}
   \right.
\]
\end{theorem}
\section{Main Results}
Let $(G,\cdot, *)$ be a multiplicative Lie algebra. Then by the universal property of non abelian exterior square of a group, we have a unique homomorphism $\phi$ from $G\wedge G$ to $G$ defined by $\phi(g\wedge h)=g*h$ and a short exact sequence of groups
\begin{center}
\begin{tikzcd}
  \{1\}\arrow{r} & \text{Ker}(\phi)\arrow{r} & G\wedge G \arrow{r}{\phi} & G*G\arrow{r} & \{1\}\cdots (1)
\end{tikzcd}
\end{center}
Also we have an epimorphism $\chi : G\wedge G \longrightarrow [G, G]$ defined by $\chi(g\wedge h)= [g, h]$ and kernel of $\chi$ is the Schur multiplier $M(G)$ of $G$, that is, we have following short exact sequence 
\begin{center}
\begin{tikzcd}
  \{1\}\arrow{r} & M(G)\arrow{r} & G\wedge G \arrow{r}{\chi} & \text{[$G$,$G$]} \arrow{r} & \{1\}\cdots (2)
\end{tikzcd}
\end{center}
From the exact sequence $(2)$, it is clear that if $G$ is abelian, then $M(G)=G\wedge G$. Also if the Schur multiplier $M(G)=\{1\}$, then $G\wedge G= [G,G]$ and so order of $G*G$ divides the order of $[G,G]$. Now we prove the following theorem which says that every multiplicative Lie algebra structure $(*)$ on a group $G$ can be uniquely determined by a homomorphism from $G\wedge G$ to $G$ under certain conditions.
\begin{theorem} \label{Wedge}
Let $G$ be a group. A homomorphism $\phi$ from $G\wedge G$ to $G$ defines a multiplicative Lie algebra structure $(*)$ on $G$ by putting $x*y=\phi(x\wedge y)$ if and only if the normal subgroup $J$ generated by the elements $\{(\phi(x\wedge y)\wedge ^yz)(\phi(y\wedge z)\wedge ^zx)(\phi(z\wedge x)\wedge ^xy)\mid x,y,z \in G \}$ belongs to the kernal of $\phi$ and $^z{\phi(x\wedge y)}=\phi(^zx\wedge ^zy),~\forall ~x,y,z\in G$. In other words, any multiplicative Lie algebra structure $(*)$ on $G$ can be uniquely determined by a homomorphism $\bar{\phi}$ from $\frac{G\wedge G}{J}$ to $G$ satisfying $^z{\bar{\phi}((x\wedge y)J)}=\bar{\phi}((^zx\wedge ^zy)J).$
\end{theorem}
\begin{proof}
Let $\phi$ be a group homomorphism from $G\wedge G$ to $G$ such that $(G,\cdot,*)$ be a multiplicative Lie algebra, where $*$ is given by $g*h=\phi(g\wedge h)$. Then
 
$\phi((\phi(x\wedge y)\wedge ^yz)(\phi(y\wedge z)\wedge ^zx)(\phi(z\wedge x)\wedge ^xy))=$
$\phi(((x*y)\wedge ^yz)((y*z)\wedge ^zx)((z*x)\wedge ^xy))=\phi(((x*y)\wedge ^yz))\phi(((y*z)\wedge ^zx))\phi(((z*x)\wedge ^xy))=((x*y)* ^yz)((y*z)* ^zx)((z*x)* ^xy)=1$ (by identity (4) of Definition \ref{MLA}). In turn, $J\subseteq $Ker$(\phi)$.

Also, $^z{\phi(x\wedge y)}=^z{(x*y)}=^zx*^zy=\phi(^zx\wedge ^zy),~\forall~x,y,z\in G$ (by identity (5) of Definition \ref{MLA}).

Conversely, suppose $J \subseteq~ \text{Ker}(\phi)$ and $^z{\phi(x\wedge y)}=\phi(^zx\wedge ^zy),~\forall~x,y,z\in G$. Define a map $*$ from $G\times G$ to $G$ by $g*h=\phi(g\wedge h)$.

\textbf{Claim:} $*$ is a multiplicative Lie algebra structure on $G$.
\begin{enumerate}
\item $x*x=\phi(x\wedge x)=\phi(1)=1$.
\item $xy*z=\phi(xy\wedge z)=\phi(^xy\wedge ^xz)\phi(x\wedge z)=~ ^x{\phi(y\wedge z)}\phi(x\wedge z)=~ ^x{(y*z)}(x*z)$ 
\item $x*yz=\phi(x\wedge yz)=\phi(x\wedge y)\phi(^yx\wedge ^yz)=\phi(x\wedge y)^y{(\phi(x\wedge z))}=(x*y)^y{(x*z)}$.
\item For all  $x,y,z\in G$, we have 
$(x*y)*^yz=\phi(x\wedge y) * ^yz=\phi(\phi(x\wedge y)\wedge ^yz)$. Hence  
 $(((x*y)* ^yz)((y*z)* ^zx)((z*x)* ^xy))$
$=\phi((\phi(x\wedge y)\wedge ^yz)(\phi(y\wedge z)\wedge ^zx)(\phi(z\wedge x)\wedge ^xy))=1$ (since $J \subseteq~ \text{Ker}(\phi)$).

\item $^z{(x*y)}=~ ^z{\phi(x\wedge y)}=\phi(^zx\wedge ^zy)=(^zx*^zy)$.
\end{enumerate}
This shows that $*$ is a multiplicative Lie algebra structure on $G$. 
\end{proof}
\begin{remark} 
\begin{enumerate}
\item If $G$ is a group with trivial Schur multiplier, then $G\wedge G \cong [G,G]$. So we can write $g\wedge h$ by $[g,h]$. Moreover, $\phi(x\wedge y)\wedge ^yz$= $[\phi([x, y], ^yz]=\phi([[x,y],^yz])$ (from the proof of Proposition $2.6$ of \cite{RLS}). Therefore $((\phi(x\wedge y)\wedge ^yz)(\phi(y\wedge z)\wedge ^zx)(\phi(z\wedge x)\wedge ^xy))=\phi([[x,y],^yz][[y,z],^zx][[z,x],^xy])=1$. So, $J$ is trivial subgroup of $G\wedge G\cong [G,G]$ and $^z\phi(x\wedge y)=\phi(^zx\wedge ^zy)\Rightarrow ^z\phi([x,y])=\phi(^z[x,y])$ which shows  that $\phi$ is a $G$-equivariant homomorphism from $[G,G]$ to $G$. It turns out that any multiplicative Lie algebra structure can be uniquely obtained from a $G-$ equivariant homomorphism from $[G,G]$ to $G$ which gives Proposition $\ref{Equivariant homomorphism}$. 
\item If $G$ is an cyclic group, that is, abelian group with trivial Schur multiplier then $G\wedge G$ is the trivial group, so there is only one homomorphism form $G\wedge G$ to $G$. Therefore $G$ has only trivial multiplicative Lie algebra structure.
\end{enumerate}
\end{remark}
\begin{corollary}
Let $G$ be an abelian group with the Schur multiplier $\mathbb{Z}_p$, where $p$ is a prime. Then $G$ has exactly two distinct multiplicative Lie structures.
\end{corollary}
\begin{proof}
Since $G$ is abelian, $ G\wedge G\cong \mathbb{Z}_p$.
By Theorem $\ref{Wedge}$, any multiplicative Lie algebra structure can be uniquely obtained from a group homomorphism $\phi$ from $\mathbb{Z}_p$ to $G$ with $J \subseteq~ \text{Ker}(\phi)$ ($J$ is defined in Theorem $\ref{Wedge}$). Since $\{ 1\}$ and  $\mathbb{Z}_p$ are only two subgroups of $\mathbb{Z}_p$, $J$ can be either trivial or $\mathbb{Z}_p$. Suppose $J$ is $\mathbb{Z}_p$. Then $G$ has only trivial multiplicative Lie algebra structure which is not possible since $G$ is not a cyclic group. Hence $J$ must be trivial. Again since there are just two distinct homorphisms from $\mathbb{Z}_p$ to $G$, $G$ has only two distinct multiplicative Lie algebra structures. 
\end{proof}

Now we calculate the number of distinct multiplicative Lie algebra structures on some finite groups.
\begin{theorem}
Let $G$ be a non-abelian group of order $pq$ such that $p$ and $q$ are distinct primes. Then $G$ is Lie simple.
\end{theorem}
\begin{proof}
By Theorem $\ref{Schur multiplier of p cyclic groups}$, the Schur multiplier of $G$ is trivial. Therefore from Proposition $\ref{Equivariant homomorphism}$, every multiplicative Lie algebra structure on $G$ can be obtained from a $G-$ equivariant homomorphism from $[G,G]$ to $G$. Suppose $p\mid (q-1)$. So $[G,G]\cong \mathbb{Z}_q$. Since there are only two distinct $G-$equivariant homomorphism from $\mathbb{Z}_q$ to $G$, $G$ is Lie simple.
\end{proof}
If gcd$(m,n)=1$, then $\mathbb{Z}_m \times \mathbb{Z}_n$ will be cyclic and become Lie simple. In the next result, we will see the characterization of multiplicative Lie algebra structures of $\mathbb{Z}_m \times \mathbb{Z}_n$ for non co-prime integers $m$ and $n$. Now onward $\tau(n)$ denotes the number of positive divisors of $n$.
\begin{theorem}
Let $G=\mathbb{Z}_m \times \mathbb{Z}_n$ be an abelian group of order $mn$ with $gcd(m,n)=d~ (\neq 1)$. Then it has $\tau(d)$ distinct multiplicative Lie algebra structures.
\end{theorem}
\begin{proof}
Suppose $G=\langle a,b:a^m=b^n=1,~ab=ba \rangle$. Since $G$ is an abelian group, by Theorem $\ref{Schur multiplier 2}$ and Theorem $\ref{tensor product}$, $G\wedge  G \cong M(G)\cong \mathbb{Z}_d$. Let $\phi:G\wedge  G\to G$ be a group homomorphism given by  $\phi(a\wedge b)=a^ib^j$, for some suitable $1\leq i \leq m$ and $1\leq j \leq n.$ and $x,y,z \in G$ such that $x=a^kb^l$, $y=a^pb^q$ and $z=a^rb^s$ for some positive integers $k,l,p,q,r \ \text{and}\ s$. Consider the expressions 

$\phi(x\wedge y)\wedge z=\phi((a^kb^l)\wedge (a^pb^q))\wedge a^rb^s=(\phi(a^k\wedge b^q)\phi(b^l\wedge a^p))\wedge a^rb^s$

$=((a^ib^j)^{kq}(a^ib^j)^{-lp})\wedge a^rb^s=((a^i)^{(kq-lp)}(b^j)^{(kq-lp)})\wedge a^rb^s$

$=(a^{i(kq-lp))}\wedge b^s)(b^{j(kq-lp))}\wedge a^r)=(a\wedge b)^{(kq-lp)(is-jr)}$

Similarly, 
$\phi(y\wedge z)\wedge x=\phi(a^pb^q\wedge a^rb^s)\wedge a^kb^l=(a\wedge b)^{(ps-qr)(il-jk)}$

and, $\phi(z\wedge x)\wedge y=\phi(a^rb^s\wedge a^kb^l)\wedge a^pb^q=(a\wedge b)^{(rl-sk)(iq-jp)}$. Hence $(\phi(x\wedge y)\wedge z)(\phi(y\wedge z)\wedge x(\phi(y\wedge z)\wedge x))=1$.

From the above calculation, it is clear that for any homomorphism $\phi:G\wedge  G\to G$,  the normal subgroup $J$ generated by the elements $\{(\phi(x\wedge y)\wedge ^yz)(\phi(y\wedge z)\wedge ^zx)(\phi(z\wedge x)\wedge ^xy)\mid x,y,z \in G \}$ is always trivial. So by Theorem $\ref{Wedge}$, every homomorphism from $G\wedge G$ to $G$ gives a multiplicative Lie algebra structure on $G$. Since there are $\tau(d)$ distinct homomorphisms from $\mathbb{Z}_d$ to $\mathbb{Z}_m \times \mathbb{Z}_n$, the number of distinct multiplicative Lie algebra structures on $G$ is equal to $\tau(d)$
\end{proof}
\begin{example}
The Klein's four group $V_4= \langle a,b:\ a^2=b^2=(ab)^2=1 \rangle$ is an abelian group of order $4$, there are two distinct multiplicative Lie algebra structure $(*)$ on $V_4$ given by  $a*b= 1 ~ \text{and} ~ a*b=a$.
\end{example}

\begin{theorem} \label{Structure on D_n}
The number of distinct multiplicative Lie algebra structures on the dihedral group $D_n=\langle  a,b:a^2=b^n=1, aba=b^{-1} \rangle$ is $\tau(n)$.
\end{theorem}
\begin{proof}
Since $D_n\wedge D_n \cong \mathbb{Z}_n = \langle a\wedge b \rangle$, we claim that for each $i=1,2,\ldots n$, $\phi_i:\mathbb{Z}_n\to D_n$ defined by $\phi_i(a\wedge b)=b^i$ is a group homomorphism such that $J\subseteq$Ker$(\phi_i)$ and $^z\phi_i(x\wedge y)=\phi_i(^zx\wedge ^zy)$ for all $x,y,z\in D_n$. It is clear that $\phi_i$ is a group homomorphism for each $i=1,2,\ldots n$. Now to show that $J\subseteq $Ker$(\phi_i)$, let $x=a,~y=b,~z=ab^j, j=1,\ldots n$, then 

$\phi_i(x\wedge y)\wedge ^yz=\phi_i(a\wedge b)\wedge ^bab^j=b^i\wedge bab^jb^{n-1}=b^i\wedge ab^{j+n-2}=b^i\wedge a$ 

$\phi_i(y\wedge z)\wedge ^zx=\phi_i(b\wedge ab^j)\wedge ^{ab^j}a=b^{n-1}\wedge ab^{2j}=b^{n-i}\wedge a$ and

$\phi_i(z\wedge x)\wedge ^xy=\phi_i(ab^j\wedge a)\wedge ^ab=\phi_i(^ab^j\wedge a)\wedge aba=1$.
Now, $\phi_i((b^i\wedge a)(b^{n-i}\wedge a))=1$ and $^{ab^j}\phi_i(a\wedge b)=ab^jb^iab^j=b^{n-i}=\phi_i(^{ab^j}a\wedge ^{ab^j}b)$. Similarly we can show that for each element of $D_n$, the conditions of Theorem $\ref{Wedge}$ are satisfied for the homomorphism $\phi_i$. Since there are only $\tau(n)$ distinct homomorphism from $\mathbb{Z}_n$ to $D_n$, $D_n$ has $\tau(n)$ distinct multiplicative Lie algebra structures. 
\end{proof}
\begin{example}
\begin{enumerate}
\item On $D_3$, there are only two distinct multiplicative Lie algebra structures $(*)$ given in usual way $a*b=1 ~ \text{and}~a*b= aba^{-1}b^{-1}$. In fact, it is Lie simple.
\item On $D_4$, we can define distinct multiplicative Lie algebra structures in following ways, $a*b=1, ~ a*b=b~\text{and}~a*b=b^2$. It can be verified from Theorem $\ref{Structure on D_n}$ that in each case $(D_4,\cdot,*)$ forms a multiplicative Lie algebra and the ideal $D_4*D_4$ is isomorphic to following subgroups $\lbrace 1 \rbrace, ~ \lbrace 1,b,b^2,b^4\rbrace ~ \text{and}~[D_4,D_4]$ of $D_4$ respectively. Therefore, $D_4$ has exactly $\tau(4)$ multiplicative Lie structures.

\end{enumerate} 
\end{example}
\begin{remark}
For odd $n$, $M(D_n)$ is trivial (Theorem $\ref{Schur multiplier of $D_n$})$. Therefore by using Proposition $\ref{Equivariant homomorphism}$, we can give a different proof of Theorem \ref{Structure on D_n}.
Let $\phi$ be a homomorphism from $[D_n,D_n]$ to $D_n$ and $g ~ (=a^ib^j) \in D_n, \ h= b^k\in [D_n,D_n]$. Therefore

\[   
ghg^{-1} = 
     \begin{cases}
       b^k ,&\quad i \ \text{is even}\\
       b^{-k}, &\quad i \ \text{is odd} \\ 
     \end{cases}
\]
It turns out that $\phi$ is a $D_n-$ equivariant homomorphism. Therefore every homomorphism from $[D_n,D_n]$ to $D_n$ gives a multiplicative Lie algebra structure. Further it can be seen that $Hom(\mathbb{Z}_n,D_n)\cong Hom(Z_n,Z_n)$. Hence, we conclude that $D_n$ has $\tau(n)$ distinct multiplicative Lie algebra structure. 
\end{remark}
\begin{theorem}
The alternating group $A_n$ is Lie simple for all $n$.
\end{theorem}
\begin{proof} 
For $n\leq 3$, it is trivial. Now take $n=4$. Let $*$ be a multiplicative Lie algebra structure on $A_4$. Since $A_4* A_4$ is a normal subgroup of $A_4$, it can be equal to either $A_4$ or $[A_4, A_4]$ or the trivial subgroup. 
We know that $A_4\wedge A_4$ is isomorphic to $Q_2$ and order of $A_4*A_4$ divides the order of $A_4\wedge A_4$. So, $A_4*A_4$ can not be $A_4$. Therefore $A_4$ is Lie simple. 

Since $A_n$ is simple group for $n\geq 5$, by Proposition $\ref{Lie simple 3}$,  $A_n$ is also Lie simple.  
\end{proof}
\begin{theorem}
The symmetric group Sym$(n)$ is Lie simple for all $n$.
\end{theorem}
\begin{proof} Let $*$ be a multiplicative Lie algebra structure on Sym$(n)$. To prove the theorem, it is enough to show that $\text{Sym}(n)* \text{Sym}(n)$ can be only either $\{1\}$ or $A_n$. We divide the proof in four cases:

\textbf{Case $(1):$} For $n=1,2$, it is clear.

\textbf{Case $(2):$} For $n=3$, the Schur multiplier $M(\text{Sym}(3))$ is trivial. Therefore any multiplicative Lie algebra structure on $\text{Sym}(3)$ can be uniquely determined by a $\text{Sym}(3)-$ equivariant homomorphism from $A_3$ to $\text{Sym}(3)$. Since there are only two distinct homomorphism from $A_3$ to $\text{Sym}(3)$, $\text{Sym}(3)$ is Lie simple.

\textbf{Case $(3):$} Now for $n \geq 5,~\text{Sym}(n)$ has only one proper normal subgroup $A_n$. Therefore to show that $\text{Sym}(n)$ is Lie simple it is enough to show that $\text{Sym}(n)* \text{Sym}(n)$ can not be equal to $\text{Sym}(n)$. Suppose on contrary. Then $\text{Sym}(n)\wedge \text{Sym}(n)$ has to be isomorphic to $\text{Sym}(n)$ which is not possible since $\text{Sym}(n)\wedge \text{Sym}(n)$ has a normal subgroup of order $2$ that can be seen by the following short exact sequence of groups
\begin{center}
\begin{tikzcd}
  \{1\}\arrow{r} & \mathbb{Z}_2\arrow{r} & \text{Sym}(n)\wedge \text{Sym}(n) \arrow{r} & A_n \arrow{r} & \{1\}
\end{tikzcd}
\end{center}
but Sym$(n)$ has no normal subgroup of order $2$. Therefore $\text{Sym}(n)* \text{Sym}(n)$ can not be equal to $\text{Sym}(n)$. Hence Sym$(n)$ is Lie simple for all $n\geq 5.$ 

\textbf{Case $(4):$} Now for $n =4,~\text{Sym}(4)$ has two proper normal subgroup $A_4$ and $V_4$. Therefore to show that $\text{Sym}(n)$ is Lie simple it suffices to prove that $\text{Sym}(n)* \text{Sym}(n)$ can not be equal to $\text{Sym}(n)$ and $V_4$. By a similar argument of \textbf{Case (3)}, we can see that Sym$(4)*$Sym$(4)$ can not be Sym$(4)$. 

Finally, we show that Sym$(4)*\text{Sym}(4)$ can not be isomorphic to $V_4$. Suppose on contrary. then there exists a surjective group homomorphism from $\text{Sym}(4)\wedge \text{Sym}(4)$ to $V_4$. But this leads to a contradiction since $\text{Sym}(4)\wedge \text{Sym}(4)\cong \text{SL}(2,3)$ and $\text{SL}(2,3)$ has no normal subgroup of order $6$. Therefore, Sym$(4)*\text{Sym}(4)$ can not be isomorphic to $V_4$.

Hence Sym$(n)$ is Lie simple for all $n$.
\end{proof}
\begin{theorem}\label{Quarternian group}
The Quarternian group $Q_{n}=\langle x,y:x^2=y^n,xyx^{-1}=y^{-1}\rangle$ of order $4n$ has $\tau(n)$ distinct multiplicative Lie algebra structures.
\end{theorem}
\begin{proof}
Since the Schur multiplier of $Q_{n}$ is trivial (Theorem $\ref{Schur multiplier of quaternion group}$), every multiplicative Lie algebra structure on $Q_n$ can be obtained from a $Q_{n}$- equivariant homomorphism from $[Q_{n},Q_{n}]$ to $Q_{n}$. From the presentation of $Q_{n}$, it can be seen that $x^4=y^{2n}=1$. Now we will show that every homomorphism $\phi$ from $[Q_{n},Q_{n}]~\text{to}~ Q_{n}$ is $Q_{n}-$ equivariant. Since $[Q_{n},Q_{n}]=\langle y^2 \rangle$, any $g\in [Q_{n},Q_{n}]$ is some power of $y^2$. If $\phi(y^2)=y^{m}$ is a homomorphism, then we have $\phi(xy^2x^{-1})=\phi(y^{-2})=y^{-m}$ and $x\phi(y^2)x^{-1}=xy^{m}x=y^{-m}$, which shows that $\phi$ is $Q_{n}$- equivariant homomorphism. Therefore every distinct homomorphism from $[Q_n,Q_n]$ to $Q_n$ gives a distinct multiplicative Lie structure on $Q_n$. Hence $Q_{n}$ has $\tau(n)$ multiplicative Lie structure.
\end{proof}
\begin{theorem}
Let $G$ be a non abelian group of order $8p$, $p$ is an odd prime with trivial Schur multiplier group. Then $G$ is Lie simple.
\end{theorem}
\begin{proof}
Since the Schur multiplier of $G$ is trivial, $G$ is isomorphic to exactly one group from $(2),(6),(8),(10),(12),(14),~\text{and}~ (16)$ in Theorem \ref{Structure on 8p}

Let $G=Q_2\times \mathbb{Z}_p$. Then $[G,G]\cong \mathbb{Z}_2$ which shows that $G$ is Lie simple.

If $G=\langle x,y:x^8=y^p=1,x^{-1}yx=y^{-1} \rangle$, then first we show that $[G,G]= \langle y \rangle$.
Since $\frac{G}{\langle y \rangle}$ is abelian, $[G,G]\leq \langle y \rangle$ and $[x^{-1},y]=x^{-1}yxy^{-1}=y^{-2}\in [G,G]\Rightarrow[G,G]=\langle y \rangle~ (\cong \mathbb{Z}_p)$. So the only normal subgroups of $[G,G]$ are $\{1\}$ and $[G,G]$ itself which forces  $G=\langle x,y:x^8=y^p=1,x^{-1}yx=y^{-1} \rangle$ to be Lie simple. 

Proceeding in a similar manner we can show that the groups $(\ref{10})$ and $(\ref{12})$ are also Lie simple. If $G=Q_{2p}$, then it follows from Theorem $\ref{Quarternian group}$. 

Now if $G=\langle a,b,c,d:a^4=b^2=c^2=d^p=1,~ab=ba, ac=ca, bc=cb,d^{-1}ad=b, d^{-1}bd=c,d^{-1}cd=ab \rangle$, then its commutator $[G,G]$ is generated by $a,b$ and $c$. Let $H$ be the subgroup of $G$ generated by the elements $a,b~\text{and}~c$, then clearly $H\cong \mathbb{Z}_2\times \mathbb{Z}_2\times \mathbb{Z}_2$ and $G/H \cong \langle d \rangle$. Therefore $[G,G]\leq H$. Also $[d^{-1},a]=ba,~[d^{-1},b]=cb~\text{and}~[d^{-1},c]=abc$. From which it clear that $[G,G]=H$.

 
Now to show that $G$ is  Lie simple, we need to show that there does not exist any $G-$equivariant homomorphism $\phi$ from $[G,G]$ to $G$ such that Im~$\phi$ is a subgroup of order $2$ or $4$. Let $\phi:[G,G]\rightarrow G$ be a $G-$ equivariant homomorphism which is not a monomorphism. Since $\phi$ is $G-$ equivariant homomorphism, $d^{-1}\phi(a)d=\phi(d^{-1}ad)=\phi(b)$ and $d^{-1}\phi(c)d=\phi(d^{-1}cd)=\phi(ab)\Rightarrow \phi(c)=d\phi(ab)d^{-1}$. Now if $\phi(a)=x,~\text{then}~\phi(b)=d^{-1}xd~\text{and}~\phi(c)=d(xdxd^{-2})$. Suppose Im~$\phi$ is  a subgroup of order $4$, we may conclude that $\circ(\phi(a))=\circ(\phi(b))\neq 1~\text{and}~\circ(\phi(c))=1$. Therefore $dxdxd^{-2}=1 \Rightarrow dxdx=d^{2}$. From the presentation of $G$, it is clear that $x=1$, which is a contradiction. Hence Im~$\phi$ can not be a subgroup of order $4$. It is also easy to see that Im~$\phi$ can not be a subgroup of order $2$.

Therefore $G=\langle a,b,c,d:a^4=b^2=c^2=d^p=1,~ab=ba, ac=ca, bc=cb,d^{-1}ad=b, d^{-1}bd=c,d^{-1}cd=ab \rangle$ is also Lie simple. 

Now consider the group $SL(2,3)=\langle x,y,z:x^4=z^{3}=1,x^2=y^2,yxy^{-1}=x^{-1},zxz^{-1}=y,zyz^{-1}=xy \rangle$. Then it is well known that $[\text{SL}(2,3),\text{SL}(2,3)]\cong Q_8$. Therefore order of $SL(2,3)*SL(2,3)$ can be either $1$, $2$, $4$ or $8$. Suppose order of $SL(2,3)*SL(2,3)$ is $4$, then it will be isomorphic to either $\mathbb{Z}_4$ or Klein Four-group $V_4$. Since $Q_8/Z(Q_8)$ can not be cyclic, it can not be $\mathbb{Z}_4$ and $V_4$ is also not possible as $SL(2,3)$ has only one element of order $2$. Further if $\phi:[\text{SL}(2,3),\text{SL}(2,3)]\rightarrow \text{SL}(2,3)$ is a $\text{SL}(2,3)-$ equivariant homomorphism, then $z\phi(x)z^{-1}=\phi(y)$ which shows that $\circ (\phi(x))=\circ (\phi(y))$ (here it is easy to observe that $x,y \in [\text{SL}(2,3),\text{SL}(2,3)]$). Therefore Im $\phi$ can not isomorphic to any subgroup of order $2$ which asserts that order of $SL(2,3)* SL(2,3)$ can not be $2$. Hence SL$(2,3)$ is also Lie simple. 

\end{proof}

\textbf{Acknowledgement}: We are extremely thankful to Prof. Ramji Lal for his valuable suggestions, discussions and constant support. The first named author thanks IIIT Allahabad for providing institute fellowship.

\end{document}